\documentclass[12pt,oneside]{amsart}
\usepackage{amssymb}
\usepackage{amsmath}
\usepackage{amsthm}
\usepackage{amscd}

\usepackage{longtable}

\theoremstyle{plain}
\newtheorem{theorem}{Theorem}
\newtheorem{lemma}{Lemma}
\newtheorem{corollary}{Corollary}

\theoremstyle{definition}

\theoremstyle{remark}

\begin{document}
\title
[The First Correction Term in Bohr--Sommerfeld Lagrangian States]
{The First Correction Term in the Asymptotic Expansion of Bohr--Sommerfeld Lagrangian States}
\author{Yusaku Tiba}
\date{}

\begin{abstract}
Let $\Lambda$ be a compact Bohr--Sommerfeld Lagrangian submanifold of a compact K\"ahler manifold equipped with a holomorphic prequantum line bundle. 
We study the asymptotic expansion of the Lagrangian states associated with $\Lambda$. 
In particular, we compute explicitly the first nontrivial correction term and show that it is expressed in terms of geometric invariants of the ambient K\"ahler manifold and the Lagrangian submanifold, including their scalar curvatures, the second fundamental form, and the mean curvature. 
As a consequence, we obtain the corresponding second-order asymptotic formula for the $L^2$-norm of the Lagrangian states. 
\end{abstract}

\maketitle



\section{Introduction}\label{section:1}
Let $(M, \omega)$ be a compact K\"ahler manifold of complex dimension $n$ with complex structure $J$.  
Let $(L, h)$ be a holomorphic line bundle over $M$ with a Hermitian metric $h$, 
and let $\nabla$ be the Chern connection of $(L, h)$.  
We denote the curvature form of $(L, \nabla)$ by $F^{\nabla}$.  
Throughout this paper, we assume $(L, \nabla)$ is a holomorphic prequantum line bundle, that is, 
$F^{\nabla} = -2\sqrt{-1} \omega$.    
For each positive integer $k$, let
$
L^k:=L^{\otimes k}
$
denote the $k$-th tensor power of $L$. 
We denote by $L^2(M,L^k)$ the Hilbert space of square-integrable sections of $L^k$, equipped with the norm
$
\|s\|_{h^k}^2
:=
\int_M |s(z)|_{h^k}^2 \frac{\omega^n}{n!}.  
$
Let
$
\Pi_k:L^2(M,L^k)\longrightarrow H^0(M,L^k)
$
be the orthogonal projection onto the space of holomorphic sections. We denote the Schwartz kernel of $\Pi_k$, called the Bergman kernel, by $K_k(z,w)$. 
The Bergman kernel $K_k$ is a smooth section of
$
L^k\boxtimes \overline{L^k}
$
over $M\times M$.

The asymptotic expansion of the Bergman kernel as $k\to\infty$ has been extensively studied. 
In particular, the second coefficient in its diagonal asymptotic expansion is equal to $1/8$ times the scalar curvature of $M$(cf.~\cite{Ber-Ber-Sjo}).

The aim of this paper is to establish an analogous result for the Lagrangian states introduced below. More precisely, we show that the second term in their asymptotic expansion can also be expressed in terms of geometric invariants associated with $M$ and the Lagrangian submanifold $\Lambda$.

Let $\Lambda$ be a compact Lagrangian submanifold of $M$. 
Since
$
\left.F^\nabla\right|_{\Lambda}=-2\sqrt{-1}\left.\omega\right|_{\Lambda}=0,
$
the restriction $\left.(L,\nabla)\right|_{\Lambda}$ is a flat Hermitian line bundle. 
We call $\Lambda$ a \emph{Bohr--Sommerfeld Lagrangian submanifold} if $\left.(L,\nabla)\right|_{\Lambda}$ is trivial as a flat line bundle, or equivalently, if it admits a nowhere-vanishing parallel section. 
This condition is the K\"ahler analogue of the Bohr--Sommerfeld condition for Lagrangian submanifolds of cotangent bundles arising in WKB analysis.
Let $\zeta$ be the parallel unit section of $L|_{\Lambda}$ on $\Lambda$.  
For a function $f$ on $\Lambda$, we define the associated Lagrangian state $s_{f,k}$ by
\[
s_{f,k}(z):=\left(\frac{\pi}{2k}\right)^{n/2} \int_{\Lambda} K_{k}(z, x)f(x) \zeta^{k}(x)dv_{\Lambda}(x).
\]
Here, $dv_{\Lambda}$ is the volume density on $\Lambda$, 
and the prefactor
$
\left(\frac{\pi}{2k}\right)^{n/2}
$
is introduced as a normalization factor.
The family $\{s_{f,k}\}_{k\in \mathbb{N}}$ forms a sequence of holomorphic sections that becomes asymptotically localized along $\Lambda$ as $k\to\infty$, 
and this construction may be regarded as the K\"ahler analogue of a WKB state in the cotangent bundle setting. 

The Lagrangian state admits a natural interpretation in terms of the
restriction operator. 
Using the parallel unit section $\zeta$, 
we identify
$L^k|_{\Lambda}$ with the trivial line bundle over $\Lambda$, and define
$
R_k:H^0(M,L^k) \longrightarrow L^2(\Lambda)
$
by
$
R_k s:=\frac{s|_{\Lambda}}{\zeta^k}.
$
The adjoint operator
$
R_k^*:L^2(\Lambda)\longrightarrow H^0(M,L^k)
$
is given by
$
R_k^*f(z)
=
\int_{\Lambda} K_k(z,x)f(x)\zeta^k(x)\,dv_{\Lambda}(x)$.  
Therefore, the Lagrangian state defined above can be written as
$
s_{f,k}
=
\left(\frac{\pi}{2k}\right)^{n/2}R_k^*f.
$
Since the restriction operator $R_k$ is not an isometry, the composition
$R_kR_k^*$ is not the identity operator on $L^2(\Lambda)$. 
Nevertheless, after the normalization appearing in the definition of $s_{f,k}$, one has
\begin{equation*}
\left(\frac{\pi}{2k}\right)^{n/2}R_kR_k^*
=
I+O(k^{-1})
\end{equation*}
as $k\to\infty$. 
In fact, it is known that this operator admits a full asymptotic expansion in powers of $k^{-1}$ (e.\,g. \cite{Bor-Pau-Uri}, \cite{Deb-Pao}, \cite{Ioo}, \cite{Pao}). 
Our main result gives an explicit formula for the
coefficient of $k^{-1}$ in this asymptotic expansion in terms of
geometric invariants of $M$ and $\Lambda$.

Before stating our main result, we describe the coordinate system used in the formulation of the expansion.  
Let $p\in \Lambda$, and let $X_1,\ldots,X_n$ be a local orthonormal frame of $T\Lambda$ near $p$. 
Since $\Lambda$ is Lagrangian, $JX_1,\ldots,JX_n$ form an orthonormal frame of the normal bundle $N\Lambda$. 
We parametrize the normal directions to $\Lambda$ at $p$ by
\[
y=(y_1,\ldots,y_n)
\longmapsto
\exp_p\left(\sum_{j=1}^n y_jJX_j(p)\right),
\]
where $\exp$ denotes the Riemannian exponential map of the K\"ahler metric.

We choose a smooth section
$\widetilde{\zeta}$ of $L$ over a neighborhood $U$ of $\Lambda$ such that
$\widetilde{\zeta}|_{\Lambda}=\zeta$,
$\overline{\partial}\widetilde{\zeta}
$ 
vanishes to any order on $\Lambda$, and 
$
\nabla\widetilde{\zeta}=0 
$
on $\Lambda$
(cf.~Proposition~1 of \cite{Tib1}).  
In what follows, by a slight abuse of notation, we denote this extension
again by $\zeta$.
Let $f \in C^{\infty}(\Lambda)$.  
For related pointwise scaling asymptotics, see also \cite{Pao}.
In \cite{Tib2}, we obtained the following expansion theorem:
\begin{theorem}[\cite{Tib2}]\label{theorem:1}
There exist polynomials $c_{p, j}(y)$ in $y \in \mathbb{R}^n$ whose coefficients depend smoothly on $p \in \Lambda$ such that  
$c_{p, 2j+1} (0) = 0$ ($j = 0, 1, \ldots$), and 
\[
\frac{s_{f, k}}{\zeta^k}\left(\exp_p\left(\frac{1}{\sqrt{k}}\sum_{j=1}^n y_j JX_{j}(p) \right)\right)
= f(p) + \sum_{j=1}^N \frac{c_{p, j}(y)}{k^{j/2}} + o\left(\frac{1}{k^{N/2}}\right).  
\]
as $k\to\infty$.
\end{theorem}

Our main result gives a geometric formula for the first nontrivial even coefficient $c_{p,2}(0)$.
In particular, $c_{p,2}(0)$ can be expressed entirely in terms of geometric invariants associated with $M$, $\Lambda$, and the function $f$.
\begin{theorem}\label{theorem:2}
\begin{align*}
c_{p,2}(0)
=
&\left(\frac{3}{32} \mathrm{Scal}_M(p) -\frac{1}{8}\mathrm{Scal}_{\Lambda}(p)+ \frac{1}{24} |B(p)|^2 - \frac{1}{8}|H(p)|^2 \right) f(p) - \frac{1}{2} \Delta_{\Lambda} f(p) \\
& \quad- \sqrt{-1} \left(\frac{1}{4} \mathrm{div}_{\Lambda} JH(p) f(p) + \frac{1}{2} JH f (p) \right).  \\
\end{align*}
Here, $\mathrm{Scal}_{M}$ and $\mathrm{Scal}_{\Lambda}$ denote the scalar
curvatures of $M$ and $\Lambda$, respectively, $B$ denotes the second
fundamental form of $\Lambda$ in $M$, and $H$ denotes its mean curvature
vector field. 
Moreover, $\operatorname{div}_{\Lambda}$ and
$\Delta_{\Lambda}$ denote the divergence and the Laplace operator on
$\Lambda$
with respect to the induced Riemannian metric on $\Lambda$, 
respectively.  
\end{theorem}

\begin{corollary}\label{corollary:1}
If $f\in C^\infty(\Lambda)$ is real-valued, 
then
\begin{align*}
& \left(\frac{2k}{\pi}\right)^{n/2}
\|s_{f,k}\|^2_{h^k} \\
= & 
\int_\Lambda f^2\,dv_\Lambda 
+
\frac{1}{k}
\int_\Lambda
\left[
\left(
\frac{3}{32}\mathrm{Scal}_M
-\frac18\mathrm{Scal}_\Lambda
+\frac1{24}|B|^2
-\frac18|H|^2
\right)f^2
- \frac{1}{2} |\nabla f|^2 
\right]dv_\Lambda
+
O\left(\frac{1}{k^2}\right).
\end{align*}
as $k\to\infty$.
\end{corollary}

\medskip 

{\it Acknowledgment.}
This work was supported by the 
Grant-in-Aid for Scientific Research (KAKENHI No.\! 26K06823).  
The author used ChatGPT for language editing and for checking some calculations.

\section{Geometric preliminaries}\label{section:2}
In this section, we collect several geometric lemmas that will be used in the proof of the main theorem.
Throughout, $\langle\cdot,\cdot\rangle$ denotes the Riemannian metric on $M$, extended complex bilinearly to the complexified tangent bundle when necessary.
Fix $p\in\Lambda$, and let
$x=(x_1,\ldots,x_n)$ be Riemannian normal coordinates on $\Lambda$
centered at $p$. 
Let $X_1,\ldots,X_n$ be a local orthonormal frame of $T\Lambda$ that coincides with the coordinate frame $\partial/\partial x_1,\ldots,\partial/\partial x_n$ at $p$.
Since $\Lambda$ is Lagrangian,
$JX_1,\ldots,JX_n$ form a local orthonormal frame of the normal bundle.
After shrinking the neighborhood if necessary, the map
\[
\Phi(x,y)
=
\exp_x\left(\sum_{j=1}^n y_j JX_j(x)\right)
\]
defines a diffeomorphism from a neighborhood of the origin in
$\mathbb{R}^{2n}$ onto a neighborhood of $p$ in $M$.

\begin{lemma}\label{lemma:1}
Fix $p\in\Lambda$ and a positive integer $N$. 
Then there exist a sufficiently small neighborhood $U$ of $p$ in $M$ and holomorphic coordinates
$
z=(z_1,\ldots,z_n)
$
centered at $p$ on $U$ such that
\[
z\circ\Phi(x,0)=x+O(|x|^N) 
\]
as $x\to 0$.
\end{lemma}
\begin{proof}
We prove the lemma by induction.
Since $J \partial/ \partial x_j = \partial/ \partial y_j$ at $p$, 
we may choose holomorphic coordinates
$z=(z_1,\ldots,z_n)$, with
$z_j=u_j+\sqrt{-1}v_j$, centered at $p$ such that
\[
\left.\frac{\partial}{\partial u_j}\right|_p
=
\left.\frac{\partial}{\partial x_j}\right|_p,
\qquad
\left.\frac{\partial}{\partial v_j}\right|_p
=
\left.\frac{\partial}{\partial y_j}\right|_p
\qquad 
(j = 1, \ldots, n).  
\]
Then we have 
$
z\circ\Phi(x,0)=x+O(|x|^2).
$

Suppose inductively that, for some $m\geq 2$, we have chosen
holomorphic coordinates satisfying
\[
z\circ\Phi(x,0)
=
x+\sum_{|I|=m}c_Ix^I+O(|x|^{m+1}),
\]
where $I=(i_1,\ldots,i_n)$ is a multi-index,
$|I|=i_1+\cdots+i_n$,
$x^I=x_1^{i_1}\cdots x_n^{i_n}$, and
$c_I\in\mathbb{C}^n$.
Define new holomorphic coordinates by
$
z'
=
z-\sum_{|I|=m}c_Iz^I.
$
Then, 
\[
z'\circ\Phi(x,0)
=
x+O(|x|^{m+1}).
\]
This completes the induction.
\end{proof}
We call such holomorphic coordinates adapted holomorphic coordinates
of order $N$ with respect to the Riemannian normal coordinates $x$.

Let $\nabla^M$ and $\nabla^\Lambda$ denote the Levi--Civita connections of
$M$ and $\Lambda$, respectively, where $\Lambda$ is equipped with the
induced Riemannian metric.
For vector fields $X, Y, Z$ on $M$, 
we define 
\[
R^M(X,Y)Z
=
\nabla^M_X\nabla^M_Y Z
-
\nabla^M_Y\nabla^M_X Z
-
\nabla^M_{[X,Y]}Z.
\]
Let $\{\Xi_i\}_{i=1}^{2n}$ be a local orthonormal frame of $TM$.  
The scalar curvature of $M$ is defined by
$
\mathrm{Scal}_M
=
\sum_{i, j =1}^{2n}
\langle R^M(\Xi_i,\Xi_j)\Xi_j,\Xi_i\rangle.  
$
We define $R^{\Lambda}$, $\mathrm{Scal}_{\Lambda}$ in the same way using $\nabla^{\Lambda}$.

We take the Riemannian normal coordinates $x = (x_1, \ldots, x_n)$ on $\Lambda$ centered at $p$ and 
the adapted holomorphic coordinates
$z=(z_1,\ldots,z_n)$ of order $N$ with respect to $x$. 
We assume that $N$ is sufficiently large.  
In these coordinates, we write
$
E_{j} = \frac{\partial}{\partial x_j}
$, 
$
\partial_j=\frac{\partial}{\partial z_j},
\overline{\partial}_{j}=\frac{\partial}{\partial \bar{z}_j}.  
$
For a smooth function $u$, we use the notation
$u_j=\partial_j u$, $u_{\bar j}=\overline{\partial}_j u$,
and similarly for higher derivatives.
We also write the K\"ahler form as $\omega = \frac{\sqrt{-1}}{2}\sum_{i, j = 1}^n g_{i\bar{j}} dz_i \wedge d\bar{z}_j$.   
We denote 
\[
R^{\Lambda}_{ijk\ell}
=
\left\langle
R^{\Lambda}\left(
E_i, E_j\right)
E_k, E_{\ell}
\right\rangle,   
\quad 
R^{M}_{i\bar{j}k\bar{\ell}}
=
\left\langle
R^M\left(
\partial_{i}, \overline{\partial}_j
\right)
\partial_k, \overline{\partial}_{\ell}
\right\rangle.   
\]
At a point where
$\langle \partial_i, \overline{\partial}_j \rangle
=\delta_{ij}/2$, 
we have
$
\sum_{i,j=1}^n R^M_{i\bar{i}j\bar{j}} = \sum_{i, j=1}^n R^M_{i\bar{j}j \bar{i}}
=
\frac{1}{8}\mathrm{Scal}_M.
$

The second fundamental form of $\Lambda$ in $M$ is defined by 
\[ 
B(X,Y) = \nabla^M_XY-\nabla^\Lambda_XY, \qquad X,Y\in T\Lambda. 
\] 
The mean curvature vector field is defined by $H=\operatorname{tr}_{g_\Lambda}B$, where $g_{\Lambda}$ is the induced Riemannian metric on $\Lambda$. 
We define
\[
h_{ijk}
=
\left\langle
\nabla^M_{E_i}E_j,JE_k
\right\rangle
=
\left\langle
B(E_i,E_j),JE_k
\right\rangle.
\]
The tensor $h=(h_{ijk})$ is the cubic form of the Lagrangian
submanifold $\Lambda$.
Since $B$ is symmetric, we have $h_{ijk}=h_{jik}$.
Moreover, since $\Lambda$ is Lagrangian,
$\langle E_j,JE_k\rangle=0$. 
Differentiating this identity in the
$E_i$-direction and using $\nabla^M J=0$, we obtain
\[
0
=\langle \nabla^M_{E_i}E_j,JE_k\rangle 
- \langle JE_j, \nabla^M_{E_i}E_k\rangle
=
h_{ijk}-h_{ikj}.
\]
Thus, $h_{ijk}$ is symmetric in all three indices.
Since $x$ is a system of Riemannian normal coordinates, we have
\[
H(x)=\sum_{i,j=1}^n h_{iij}(x)JE_j(x)+O(|x|^2)
\]
near $p$.

\begin{lemma}\label{lemma:2}
\[
\partial_j g_{k \bar{\ell}}|_{\Lambda}
= \langle \nabla^{\Lambda}_{E_j}E_{k}, E_\ell \rangle + \sqrt{-1} h_{jk \ell} + O(|x|^{N-2}).  
\]
In particular, 
we have 
\[
\partial_j g_{k \bar{\ell}} (p) = \sqrt{-1} h_{jk \ell}(p).  
\]
\end{lemma}
\begin{proof}
We have 
\begin{align*}
\partial_{j} g_{k \bar{\ell}}
= 2\partial_{j} \langle \partial_{k}, \overline{\partial}_\ell \rangle
= 2\langle \nabla^M_{\partial_{j}} \partial_{k}, \overline{\partial}_\ell\rangle 
= 2\langle \nabla^M_{\partial_j + \overline{\partial}_j} (\partial_{k} + \overline{\partial}_{k}), \overline{\partial}_\ell \rangle.  
\end{align*}  
Because 
$z(\Phi(x, 0)) = x + O(|x|^N)$, 
it follows that 
$
E_j = \partial_j|_{\Lambda} + \overline{\partial}_j|_{\Lambda} + O(|x|^{N-1}).  
$
Here we identify $T\Lambda$ with its image in $TM|_\Lambda$.
Then, 
$
2\overline{\partial}_j|_{\Lambda}
= E_j + \sqrt{-1}J E_j + O(|x|^{N-1}).
$
Hence 
\begin{align*}
\partial_j g_{k \bar{\ell}}|_{\Lambda}
= \langle \nabla^M_{E_j} E_{k}, E_{\ell} +\sqrt{-1} J E_\ell\rangle + O(|x|^{N-2})
= \langle \nabla^{\Lambda}_{E_j}E_{k}, E_\ell \rangle + \sqrt{-1} h_{jk \ell} + O(|x|^{N-2}).  
\end{align*}  
Since $x = (x_1, \ldots, x_n)$ is the Riemannian normal coordinates, $\nabla^{\Lambda}_{E_j}E_k(p) = 0$.  
This completes the proof.  
\end{proof}

We recall the following standard identity for Riemannian normal coordinates.
\begin{lemma}\label{lemma:2-1}
\[
E_{i} \langle \nabla_{E_j}^{\Lambda}E_k, E_{\ell} \rangle 
= \frac{1}{3} (R_{ijk\ell}^{\Lambda} + R_{ikj\ell}^{\Lambda})  
\]
at $p$.  
\end{lemma}
\begin{proof}
Write $ \nabla_{E_j}^{\Lambda} E_k = \sum_{\ell=1}^n\Gamma_{j k}^{\ell} E_{\ell}$.  
Since $\nabla^{\Lambda}$ is torsion--free, $\Gamma_{jk}^{\ell}= \Gamma_{kj}^{\ell}$.  
Since $x = (x_1, \ldots, x_n)$ is the Riemannian normal coordinates centered at $p$, 
for every $x$ sufficiently close to the origin, the curve
$
t\longmapsto tx
$
is a geodesic through $p$.
Hence 
$\sum_{j, k = 1}^n \Gamma_{jk}^{\ell} (tx) x_j x_k = 0$ for $\ell = 1, \ldots, n$.  
Then, by the Taylor expansion at the origin, we have 
\[
\Gamma_{jk}^{\ell}(0) = 0, \quad E_i \Gamma_{jk}^{\ell}(0) + E_j \Gamma_{ki}^{\ell}(0) + E_k \Gamma_{ij}^{\ell}(0) = 0.  
\]
Because 
\[
R_{ijk\ell}^{\Lambda} = \langle \nabla^{\Lambda}_{E_i} \nabla^{\Lambda}_{E_j} E_{k}, E_{\ell}\rangle- \langle \nabla^{\Lambda}_{E_j} \nabla^{\Lambda}_{E_i} E_{k}, E_{\ell}\rangle
\]
at $p$, 
we have 
\[
R^{\Lambda}_{ijk\ell} = E_i \Gamma_{jk}^{\ell} - E_j \Gamma_{ik}^{\ell},  \quad 
R^{\Lambda}_{ikj\ell} = E_i \Gamma_{k j}^{\ell} - E_{k} \Gamma_{ij}^{\ell} 
\]
at $p$.  
Adding these two identities proves the lemma, since 
$E_{i} \langle \nabla_{E_j}^{\Lambda}E_k, E_{\ell} \rangle = E_i \Gamma_{jk}^{\ell}$ at $p$.  
\end{proof}

\begin{lemma}\label{lemma:3}
\begin{align*}
\sum_{i, j=1}^nR^{M}_{i\bar{j}i\bar{j}}(p) 
& = \frac{1}{8} \mathrm{Scal}_{M}(p) - \frac{1}{2} \mathrm{Scal}_{\Lambda}(p)-\frac{1}{2}|B(p)|^2 + \frac{1}{2} |H(p)|^2 
\end{align*}
\end{lemma}
\begin{proof}
At $p$, we have 
\begin{align*}
&\sum_{i, j=1}^n \langle R^M(E_i, E_j)E_j, E_i\rangle 
= \sum_{i, j=1}^n \langle R^M(\partial_i +\overline{\partial}_{i}, \partial_j + \overline{\partial}_j)(\partial_j +\overline{\partial}_{j}), \partial_i +\overline{\partial}_{i} \rangle \\
= & -2\sum_{i, j=1}^nR^{M}_{i\bar{j}i\bar{j}} + 2\sum_{i, j=1}^n R^M_{i\bar{j}j \bar{i}}
= -2\sum_{i, j=1}^nR^{M}_{i\bar{j}i\bar{j}} + \frac{1}{4} \mathrm{Scal}_{M}.  
\end{align*}
By the Gauss equation, we have 
\begin{align*}
\sum_{i, j=1}^n \langle R^{M}(E_i, E_j)E_{j}, E_i\rangle 
& = \sum_{i, j=1}^n \left(\langle R^{\Lambda}(E_i, E_j)E_j, E_i\rangle + |B(E_i, E_j)|^2 -\langle B(E_i, E_i), B(E_j, E_j) \rangle \right) \\
& =\mathrm{Scal}_{\Lambda} +|B|^2 -|H|^2 
\end{align*}
at $p$.  
This completes the proof.  
\end{proof}

\section{Asymptotic expansion of the Lagrangian state}\label{section:3}
We begin by recalling the proof of Theorem~\ref{theorem:1}. 
Since the proof of our main theorem is obtained by refining the asymptotic analysis used there, we repeat the argument in some detail and keep track of the terms needed for the computation of $c_{p,2}(0)$.

Let $p \in \Lambda$, and let $x = (x_1, \ldots, x_n)$ be Riemannian normal coordinates on $\Lambda$ centered at $p$.   
We take the map $\Phi$ and the adapted holomorphic coordinates
$z=(z_1,\ldots,z_n)$ of order $N$ with respect to $x$. 
Let $\zeta$ be the almost holomorphic extension to a neighborhood of $\Lambda$ introduced in the Introduction.  
We choose a holomorphic section $s$ of $L$ near $p$ such that $\zeta(p) = s(p)$.  

In what follows, we use the same adapted holomorphic coordinate chart
centered at $p$ on both factors of $M\times M$, and denote the corresponding
coordinates by $z$ and $w$, respectively.
All local constructions are understood in sufficiently small neighborhoods of $p$.
Define $\phi := -\log |s|_h^2$, which is a smooth plurisubharmonic function such that 
$\frac{\sqrt{-1}}{2} \partial \overline{\partial} \phi = \omega$.   
Let $\psi_0(z, \bar{w})$ be a smooth function such that $\overline{\psi_0(z, \overline{w})} = \psi_0(w, \bar{z})$, 
$\psi_0 (z, \bar{z}) = \phi(z)$, and $\overline{\partial} \psi_0$ vanishes to infinite order along the diagonal.  

The Bergman kernel $K_{k}$ is a smooth section of $L^k\boxtimes \overline{L}^k$ over $M \times M$.  
Let $\pi_j:M\times M\to M$ be the $j$-th projection, $j=1,2$. 
On a neighborhood of $(p,p)$, we write
$K_{k}/\pi_1^*s^k \pi_2^*\bar{s}^k = K_k/s^k \bar{s}^k$ for simplicity.
Then the Bergman kernel has the following expansion formula (see e.g., \cite{Ber-Ber-Sjo}): 
\begin{align*}
&\left|
\frac{K_k(z,w)}{s^k(z)\overline{s^k(w)}}
-\frac{k^n}{\pi^n}e^{k\psi_0(z,\bar w)}
\left(
1+\frac{b_1(z,\bar w)}{k}
+\cdots+
\frac{b_{\lceil N/2\rceil}(z,\bar w)}
{k^{\lceil N/2\rceil}}
\right)
\right| \\
\leq &  
e^{\frac{k}{2}\phi(z)+\frac{k}{2}\phi(w)}
O\left(
k^{-\lceil N/2\rceil-1+n}
\right),
\end{align*}
on a neighborhood of $(p,p)$, where $b_j$ are smooth functions. 
In particular,
$
b_1(z,\bar z)=\frac{1}{8}\mathrm{Scal}_M(z).
$
Let $c(z)=\frac{\zeta(z)}{s(z)}$ be a smooth function. 
Then $c(0)=1$, and $\overline{\partial}c$ vanishes to infinite order along $\Lambda$ near $p$.
Let $\varphi = -\log |\zeta|^2_h$.  
Then $\varphi(z) = \phi(z) - 2 \log |c(z)|$.  
Set $\psi(z, \bar{w}) = \psi_0(z, \bar{w}) - \log c(z) - \log \overline{c(w)}$, 
where we take the branch of the logarithm satisfying $\log 1=0$. 
Then
\begin{align*}
&\left|
\frac{K_k(z,w)}
{\zeta^k(z)\overline{\zeta^k(w)}}
-
\frac{k^n}{\pi^n}
e^{k\psi(z,\bar w)}
\left(
1+\frac{b_1(z,\bar w)}{k}
+\cdots+
\frac{b_{\lceil N/2\rceil}(z,\bar w)}
{k^{\lceil N/2\rceil}}
\right)
\right|
\\
\leq & 
e^{\frac{k}{2}\varphi(z)+\frac{k}{2}\varphi(w)}
O\left(
k^{-\lceil N/2\rceil-1+n}
\right).
\end{align*}
Since 
$\varphi \circ \Phi(x, y) = 2|y|^2 + O(|y|^3)$ and $\frac{\partial}{\partial z_{j}}\big|_{p} = \frac{1}{2} \left(\frac{\partial}{\partial x_j} -\sqrt{-1} \frac{\partial}{\partial y_j} \right)\big|_p$, 
we have 
$\varphi(z) = 2|\mathrm{Im}\, z|^2 + O(|z|^3)$. 
Since $\overline{\partial} \psi$ vanishes to infinite order along the diagonal set and $\psi(z, \bar{z}) = \varphi(z)$, 
the Taylor expansion shows that 
$
\psi(z, \bar{w}) = \frac{1}{2}(-z^2 - \bar{w}^2 +2z\bar{w})  + O(|z|^3+|w|^{3}).   
$
Let $(x',y')$ be a second copy of the coordinates $(x,y)$.
We set 
\begin{align*}
\psi_{\Phi}((x, y), (x', y')) & = \psi(z(\Phi(x, y)), \overline{z(\Phi (x', y'))}), \\
b_{\Phi, j}((x, y), (x', y')) & = b_j(z(\Phi(x, y)), \overline{z(\Phi (x', y'))}), \\
K_{k,\Phi}((x,y),(x',y')) &= K_k(\Phi(x,y),\Phi(x',y')).    
\end{align*}
We write the Taylor expansion of $\psi_{\Phi}$ as
\begin{align*}
\psi_{\Phi}((0,y),(x,0))
=  \frac{1}{2}|y|^2-\frac{1}{2}|x|^2 + \sqrt{-1}\,x\cdot y + R(y,x) + O\left(|y|^{N+3}+|x|^{N+3}\right),
\end{align*}
where
$
R(y,x)=\sum_{\ell=3}^{N+2}R_{\ell}(y,x),
$
and each $R_{\ell}(y,x)$ is a homogeneous polynomial of degree $\ell$ in $(y,x)$,
whose coefficients depend smoothly on $p$.
The precise form of these
coefficients will be important for the proof of Theorem~\ref{theorem:2} and will be
analyzed in Section~\ref{section:4}. 
For the moment, we continue the argument of \cite{Tib2}.
Replacing $y$ by $y/\sqrt{k}$ and multiplying by $k$, we have 
\begin{align*}
& k \psi_{\Phi}\left(\left(0, \frac{y}{\sqrt{k}}\right), (x, 0)\right) \\
= & \frac{1}{2} |y|^2 - \frac{k}{2}|x|^2 + \sqrt{-1}\sqrt{k} x\cdot y + kR\left(\frac{y}{\sqrt{k}}, x\right) + k O\left(\left|\frac{y}{\sqrt{k}}\right|^{N+3} + |x|^{N+3}\right).  
\end{align*}
In the Riemannian normal coordinates $x=(x_1,\ldots,x_n)$ on $\Lambda$, we write
\[
f\,dv_\Lambda
=
\left(
f(p)+\sum_{1\leq |I|\leq N}a_Ix^I+O(|x|^{N+1})
\right)
dx_1\cdots dx_n.
\]
Here $a_I \in \mathbb{C}$ is determined by the derivatives of $f$ and $\langle \frac{\partial}{\partial x_i}, \frac{\partial}{\partial x_j}\rangle$ at $p$.  
The coefficients $a_I$ needed below will also be computed explicitly in the next section.

It is known that the off-diagonal asymptotic of the Bergman kernel satisfies 
\[
|K_{k}(Z, Z')|_h \lesssim k^n e^{-c \sqrt{k}  d_{M}(Z, Z')}
\] 
for some $c > 0$, where $d_M(Z, Z')$ is the distance between $Z \in M$ and $Z' \in M$ (see Theorem~0.1 of \cite{Ma-Ma2}).  
Here $| \cdot |_h$ is the norm of $L^k \boxtimes \overline{L}^k$ induced by the Hermitian metric $h$ of $L$.  
Let $\eta:\mathbb{R} \to \mathbb{R}$ be a smooth function such that $\eta(r) = 1$ if $0 \leq r \leq 1$ and $\eta(r) = 0$ if $r \geq 2$.  
For $C_N>0$, set
$
\eta_{k,N}(x)
=
\eta\left(
\frac{\sqrt{k}|x|}{C_N\log k}
\right).
$
Taking $C_N$ sufficiently large, the off-diagonal estimate gives
\[
\left|
(1-\eta_{k, N}(x)) \frac{K_{k, \Phi}((0, \frac{y}{\sqrt{k}}), (x, 0))}{\zeta^k(\Phi(0, \frac{y}{\sqrt{k}}))\overline{\zeta^k(\Phi(x, 0))}}\right| 
= e^{\frac{k}{2} \varphi \left(\Phi(0, \frac{y}{\sqrt{k}})\right)} O\left(\frac{1}{k^{\lceil N/2 \rceil - n/2 + 1}}\right), 
\]
and we have 
\begin{align*}
&s_{f, k} \left(\Phi\left(0, \ldots, 0, \frac{y_1}{\sqrt{k}}, \ldots, \frac{y_n}{\sqrt{k}}\right)\right) \\
= & \left(\frac{\pi}{2k}\right)^{n/2} \int_{x \in \mathbb{R}^n} \eta_{k, N}(x) \frac{K_{k, \Phi}((0, \frac{y}{\sqrt{k}}), (x, 0))}{\overline{\zeta^k(\Phi(x, 0))}}  \left(f(p) + \sum_{1 \leq |I| \leq N} a_{I} x^{I} + O(|x|^{N+1}) \right) dx_1 \cdots dx_n \\
& +  \zeta^k \left(\Phi\left(0, \frac{y}{\sqrt{k}}\right)\right) e^{\frac{k}{2}\varphi\left(\Phi(0, \frac{y}{\sqrt{k}})\right)}O\left(\frac{1}{k^{\lceil N/2 \rceil +1}}\right).  
\end{align*}
For $y$ in any fixed bounded subset of $\mathbb{R}^n$, uniformly as $k\to\infty$, we have
$
k\varphi\circ\Phi\left(0,\frac{y}{\sqrt{k}}\right)
=2|y|^2+O\left(\frac{|y|^3}{\sqrt{k}}\right)
=O(1),
$
whereas
$
k\varphi\circ\Phi(x,0)=0
$
for all $x$ in the coordinate neighborhood.
We have 
\begin{align*}
& \left(\frac{\pi}{2k}\right)^{n/2}
\eta_{k, N}(x)\frac{K_{k,\Phi}\left(\left(0, \frac{y}{\sqrt{k}}\right), (x, 0)\right)}
{\overline{\zeta^k (\Phi(x, 0))}} 
\left(f(p) + \sum_{1 \leq |I| \leq N} a_{I} x^{I} + O(|x|^{N+1}) \right) dx_1 \cdots dx_n  \\
= &  
\left(\frac{k}{2\pi}\right)^{n/2} 
\eta_{k, N}(x)
\zeta^k \left( \Phi\left(0, \frac{y}{\sqrt{k}}\right) \right)
e^{\frac{|y|^2}{2} - k \frac{|x|^2}{2} + \sqrt{-1}\sqrt{k}\,x\cdot y + kR\left(\frac{y}{\sqrt{k}},x\right)+kO\left(\left|\frac{y}{\sqrt{k}}\right|^{N+3}+|x|^{N+3}\right)}\\
& \times  
\left(1 + \sum_{j=1}^{\lceil N/2 \rceil}
\frac{b_{\Phi,j}\left(\left(0, \frac{y}{\sqrt{k}}\right), (x, 0)\right)}{k^{j}} \right)  
\left(f(p) + \sum_{1 \leq |I| \leq N} a_{I} x^{I} + O(|x|^{N+1}) \right) dx_1 \cdots dx_n \\
& + \zeta^k\left(\Phi\left(0,\frac{y}{\sqrt{k}}\right)\right)
O\left(\frac{1}{k^{\lceil N/2 \rceil +1}}\right) dx_1\ldots dx_n.  
\end{align*}
By changing variables $x\mapsto x/\sqrt{k}$ and absorbing the cutoff error into the remainder, we obtain
\begin{align*} 
& \left(\frac{\pi}{2k}\right)^{n/2}
\int_{x \in \mathbb{R}^n} \eta_{k,N}(x)
\frac{K_{k,\Phi}\left(\left(0,\frac{y}{\sqrt{k}}\right),(x,0)\right)}
{\overline{\zeta^k(\Phi(x,0))}} 
\left(f(p) + \sum_{1 \leq |I| \leq N} a_{I} x^{I} + O(|x|^{N+1}) \right) dx_1 \cdots dx_n  \\
= & \zeta^k\left(\Phi\left(0,\frac{y}{\sqrt{k}}\right)\right)
\frac{e^{\frac{|y|^2}{2}}}{(2\pi)^{n/2}}
\int_{x \in \mathbb{R}^n} e^{-\frac{|x|^2}{2}+\sqrt{-1}x\cdot y}
\left(f(p) + \sum_{j=1}^{N}\frac{Q_{p,j}(y,x)}{k^{j/2}} \right) dx_1\cdots dx_n \\
& + \zeta^k\left(\Phi\left(0,\frac{y}{\sqrt{k}}\right)\right)
o\left(\frac{1}{k^{N/2}}\right).   
\end{align*}
uniformly for $y$ in any fixed bounded subset of $\mathbb{R}^n$.
Here, $Q_{p,j}(y,x)$ is a polynomial in $(y,x)$ of degree at most $3j$, determined by the Taylor coefficients of $R_{\ell}$, $b_{\Phi,j}$, and $a_I$.  
Each monomial appearing in $Q_{p,2j+1}(y,x)$ has odd total degree in $(y,x)$.  
For a multi-index $I=(i_1,\ldots,i_n)$ with $i_j\geq 0$, we have
\[
\frac{1}{(2\pi)^{n/2}}\int_{\mathbb{R}^n} x^{I} e^{-\frac{|x|^2}{2} + \sqrt{-1}x\cdot y}dx_1 \ldots dx_n 
= (-\sqrt{-1})^{|I|}\left(\frac{\partial}{\partial y}\right)^{I} e^{-\frac{|y|^2}{2}}.  
\]
Put 
\[
c_{p, j}(y) = \frac{e^{\frac{|y|^2}{2}}}{(2\pi)^{n/2}} \int_{x \in \mathbb{R}^n} e^{-\frac{|x|^2}{2}+\sqrt{-1}x\cdot y} Q_{p, j}(y, x) dx_1\cdots dx_n.   
\]
Then $c_{p,j}(y)$ is a polynomial in $y$ of degree at most $3j$. 
Moreover,
$c_{p,2j+1}(0)=0$, since each monomial in $Q_{p,2j+1}(0,x)$ has odd degree in $x$.
The coefficients of $c_{p, j}$ are smooth functions of $p$ determined by construction.  
Now we have the asymptotic expansion 
\[
\frac{s_{f, k}}{\zeta^k}  \left(\Phi\left(0, \frac{y}{\sqrt{k}}\right) \right)= f(p) + \sum_{j=1}^{N} \frac{c_{p, j}(y)}{k^{j/2}} + o\left(\frac{1}{k^{N/2}}\right).  
\]
The expansion is uniform for $y$ in any fixed bounded subset of $\mathbb{R}^n$ and $p\in\Lambda$, and this completes the proof of Theorem~\ref{theorem:1}.  

We now express the first two coefficients, $Q_{p,1}$ and $Q_{p,2}$, in terms of $a_I$, $R_{\ell}$, and $b_{\Phi,j}$.
Since 
\begin{align*}
&\exp\left(k R\left(\frac{y}{\sqrt{k}}, \frac{x}{\sqrt{k}} \right) + kO\left(\left|\frac{x}{\sqrt{k}} \right|^{N+3} + \left|\frac{y}{\sqrt{k}} \right|^{N+3}\right)\right) \\
= & 
1 + \frac{R_{3}(y, x)}{\sqrt{k}} + \frac{1}{k}\left(R_4(y, x) + \frac{R_{3}(y, x)^2}{2}\right) + O\left(\frac{1}{k^{3/2}} \right) 
\end{align*}
and 
\[
\frac{1}{k}b_{\Phi, 1}\left(\left(0, \frac{y}{\sqrt{k}} \right), \left(\frac{x}{\sqrt{k}}, 0 \right)\right) 
= \frac{\mathrm{Scal}_{M}(p)}{8k} + O\left( \frac{1}{k^{3/2}}\right),  
\]
we have 
\begin{align*}
Q_{p, 1}(y, x) & =  R_3(y, x) f(p) + \sum_{j=1}^n a_jx_j , \\ 
Q_{p, 2}(y, x) & = \left(R_4(y, x) + \frac{R_3(y, x)^2}{2} + \frac{\mathrm{Scal}_M(p)}{8}\right) f(p) + R_{3}(y, x)\sum_{j=1}^n a_j x_j + \sum_{|I| = 2} a_{I} x^{I}.  
\end{align*}
To determine $c_{p,2}(0)$ explicitly, it remains to compute the relevant Taylor coefficients appearing in $R_3$, $R_4$, and the coefficients $a_j$ and $a_I$ ($|I| = 2$).

\section{Computation of the Taylor coefficients}\label{section:4}
We now compute the Taylor coefficients needed to determine $c_{p,2}(0)$ explicitly.  
We use the same notation as in the previous section. 
Fix $p\in\Lambda$,
and keep the Riemannian normal coordinates $x$, the adapted holomorphic
coordinates $z$ of order $N$ with respect to $x$, and the map $\Phi$ introduced there.
We assume that $N$ is sufficiently large.

First, we compute $R_3(0,x)$. 
\begin{lemma}\label{lemma:4}
\[
R_{3}(0, x) = \frac{\sqrt{-1}}{3!} \sum_{i, j, k = 1}^n h_{ijk}(p) x_i x_j x_k.   
\]
\end{lemma}
\begin{proof}
By the definition of $R_3$ and the
almost holomorphicity of $\psi$, we have
\begin{align*}
R_3(0,x)
&=
\frac{1}{3!}
\sum_{i,j,k=1}^n
\frac{\partial^3\psi_\Phi}
{\partial x'_i\partial x'_j\partial x'_k}
\bigl((0,0),(0,0)\bigr)x_ix_jx_k 
=
\frac{1}{3!}
\sum_{i,j,k=1}^n
\frac{\partial^3 \psi}
{\partial w_i\partial w_j\partial w_k}(0,0)
\,x_ix_jx_k \\
& =
\frac{1}{3!}
\sum_{i,j,k=1}^n
\frac{\partial^3\varphi}
{\partial\bar z_i\partial\bar z_j\partial\bar z_k}(0)
\,x_ix_jx_k.
\end{align*}
Since $d\varphi = 0$ identically on $\Lambda$, 
$
(\partial_i + \overline{\partial}_i) (\partial_j + \overline{\partial}_j) \varphi_{\bar{k}}(0) = 0.  
$
Hence 
\[
\varphi_{ij\bar{k}} + \varphi_{i\bar{j}\bar{k}} + \varphi_{\bar{i}j\bar{k}} + \varphi_{\bar{i}\bar{j}\bar{k}} = 0
\]
at $p$ for $1 \leq i, j, k \leq n$.  
Because $g_{i\bar{j}}(z) = \varphi_{i\bar{j}}(z) + O(|z|^{\infty})$, we have 
$\varphi_{\bar{i}\bar{j}\bar{k}} = -(\partial_{i} g_{j \bar{k}} + \overline{\partial}_j g_{i\bar{k}} + \overline{\partial}_{i} g_{j\bar{k}})$ at $p$.  
Then, by Lemma~\ref{lemma:2} and the symmetry of $h_{ijk}$, 
it follows that 
$\varphi_{\bar{i}\bar{j}\bar{k}} = \sqrt{-1} h_{ijk}$ at $p$.  
Hence we complete the proof.  
\end{proof}

Next, we compute $R_4(0,x)$. 
\begin{lemma}\label{lemma:5}
\[
R_4(0,x)
=
\frac{1}{4!}
\sum_{i,j,k,\ell=1}^n
\left(
\sum_{m=1}^n h_{ikm}(p)h_{j\ell m}(p)
-2R^M_{i\bar j k\bar\ell}(p)
+2\sqrt{-1}\,E_i h_{jk\ell}(p)
\right)
x_ix_jx_kx_\ell.
\]
\end{lemma}
\begin{proof}
By the same argument as above, using the adapted holomorphic coordinates, we have
\begin{align*}
\varphi_{\bar{i}\bar{j}\bar{k}\bar{\ell}}
& = -(\varphi_{ijk\bar{\ell}} +
\varphi_{ij\bar{k}\bar{\ell}} + 
\varphi_{i\bar{j}k\bar{\ell}} + 
\varphi_{i\bar{j}\bar{k}\bar{\ell}} + 
\varphi_{\bar{i}jk\bar{\ell}} + 
\varphi_{\bar{i}j\bar{k}\bar{\ell}} + 
\varphi_{\bar{i}\bar{j}k\bar{\ell}}) \\
& = 
-(
\partial_i\partial_j g_{k\bar\ell}
+\partial_i\overline{\partial}_k g_{j\bar\ell}
+\partial_i\overline{\partial}_j g_{k\bar\ell}
+\overline{\partial}_j\overline{\partial}_k g_{i\bar \ell}
+\partial_j\overline{\partial}_i g_{k\bar\ell}
+\overline{\partial}_i\overline{\partial}_k g_{j\bar\ell}
+\overline{\partial}_i\overline{\partial}_j g_{k\bar\ell}
)
\end{align*}
at $p$, and 
\begin{align*}
R_4(0,x)
= &
\frac{1}{4!}
\sum_{i,j,k,l=1}^n
\varphi_{\bar i\bar j\bar k\bar l}(0)
x_ix_jx_kx_l\\
= & -\frac{1}{4!} \sum_{i, j, k, \ell=1}^n \left(\partial_i \partial_j g_{k\bar{\ell}}(0) + 3 \partial_i \overline{\partial}_{j} g_{k \bar{\ell}}(0) + 3\overline{\partial}_i \overline{\partial}_{j} g_{k \bar{\ell}}(0)\right) x_ix_jx_kx_{\ell}.   
\end{align*}

We first compute the mixed derivatives $\partial_i\overline{\partial}_j g_{k\bar\ell}$.
Let $G = (g_{i\bar{j}})_{1 \leq i, j \leq n}$.  
We have
\[
R^M|_{T^{(1, 0)}(M)}=\overline{\partial}({}^tG^{-1}\partial \,{}^tG) 
= -\overline{\partial}\, {}^t G \wedge \partial\, {}^t G + \overline{\partial} \partial \,{}^t G
\] 
at $p$.  
Then, 
\begin{align*}
R^M_{i\bar{j} k \bar{\ell}} = \frac{1}{2} \left(\sum_{m = 1}^n \partial_i g_{k \bar{m}} \overline{\partial}_{j}g_{m\bar{\ell}} - \partial_{i}\overline{\partial}_j g_{k \bar{\ell}}\right) 
= \frac{1}{2} \left(\sum_{m=1}^n h_{ikm}h_{jm\ell} - \partial_i \overline{\partial}_j g_{k\bar{\ell}}\right).  
\end{align*}
at $p$, and 
hence, 
\[
\partial_i \overline{\partial}_j g_{k \bar{\ell}}(p) = \sum_{m=1}^n h_{ikm}(p) h_{j\ell m}(p) - 2R^M_{i\bar{j} k \bar{\ell}}(p). 
\]

Next, we compute $\partial_i\partial_j g_{k\bar\ell}$.  
By Lemma~\ref{lemma:2}, 
\begin{align*}
\partial_j g_{k \bar{\ell}}|_{\Lambda} = \langle \nabla^{\Lambda}_{E_j}E_{k}, E_{\ell} \rangle + \sqrt{-1} h_{jk \ell} + O(|x|^{N-2}).  
\end{align*}  
Then 
\[
E_i \partial_j g_{k \bar{\ell}} = \frac{1}{3} (R_{ijk\ell}^{\Lambda} + R_{ikj\ell}^{\Lambda}) + \sqrt{-1} E_i h_{jk\ell} 
\]
at $p$ by Lemma~\ref{lemma:2-1}.  
Hence, we obtain
\begin{align*}
\partial_i \partial_j g_{k \bar{\ell}} = E_{i} \partial_j g_{k \bar{\ell}} - \overline{\partial}_{i} \partial_j g_{k \bar{\ell}} 
= \frac{1}{3} (R_{ijk\ell}^{\Lambda} + R_{ikj\ell}^{\Lambda}) + \sqrt{-1} E_i h_{jk\ell} 
- \sum_{m=1}^n h_{i\ell m}h_{j k m} +2R_{j\bar{i} k\bar{\ell}}^M 
\end{align*}
at $p$.  
Since $\overline{R^{M}_{j\bar{i} \ell \bar{k}}} = R^{M}_{i\bar{j}k\bar{\ell}}$, we have 
\[
\overline{\partial}_i \overline{\partial}_j g_{k \bar{\ell}}  
= \frac{1}{3} (R_{ij\ell k}^{\Lambda} + R_{i\ell jk}^{\Lambda}) - \sqrt{-1} E_i h_{jk\ell} 
- \sum_{m=1}^n h_{ik m}h_{j \ell m} +2R_{i\bar{j} k\bar{\ell}}^M.   
\]
We note that 
$\sum_{i,j,k,\ell = 1}^n R_{ijk\ell}^{\Lambda} x_ix_jx_kx_{\ell} = 0$, and 
\[
\sum_{i, j, k, \ell=1}^n\sum_{m= 1}^n h_{ikm}h_{j\ell m} x_i x_j x_k x_{\ell} = \sum_{i, j, k, \ell=1}^n\sum_{m= 1}^n h_{i\ell m}h_{jk m} x_i x_j x_k x_{\ell}. 
\] 
Combining the above identities, we obtain the lemma.
\end{proof}

For $1 \leq i, j \leq n$, we define $a_i$, $a_{ij} = a_{ji}$ by 
\[
fdv_{\Lambda} = \left(f(p) + \sum_{i=1}^n a_i x_i + \sum_{i, j=1}^n a_{ij} x_ix_j + O(|x|^3) \right)dx_1 \cdots dx_n.  
\]
\begin{lemma}\label{lemma:6}
\[
a_i = E_{i}f (p), \quad a_{ij} = \frac{1}{2}E_i E_j f(p) -\frac{1}{6} f(p) \sum_{m = 1}^n R^{\Lambda}_{mijm}.    
\]
\end{lemma}
\begin{proof}
By the Taylor expansion at the origin, we have 
\begin{align*}
f(x) & = f(p) + \sum_{i = 1}^n E_i f(p) x_i + \frac{1}{2} \sum_{i, j=1}^n E_{i} E_{j} f (p) x_ix_j + O(|x|^3), \\
dv_{\Lambda}(x) & = \left(1 -\frac{1}{6}\sum_{i, j, m = 1}^n R^{\Lambda}_{m ij m} x_i x_j + O(|x|^3) \right) dx_1 \cdots dx_n.  
\end{align*}
This completes the proof.  
\end{proof}

\section{Proof of Theorem~\ref{theorem:2}}\label{section:5}
We now prove Theorem~\ref{theorem:2}.
By the asymptotic expansion obtained in Section~\ref{section:3} and Lemmas in Section~\ref{section:4}, 
we have 
\begin{align*}
& Q_{p, 2}(0, x) \\
= & \frac{1}{8} \mathrm{Scal}_M(p) f(p) + \sum_{i, j = 1}^n\left(\frac{1}{2}E_i E_j f(p) - \frac{f(p)}{6} \sum_{m = 1}^n R_{m ij m}^{\Lambda}(p) \right)x_i x_j \\
& \quad + \frac{f(p)}{24} \sum_{i, j, k, \ell=1}^n \left(\sum_{m=1}^n h_{ikm}(p)h_{j\ell m}(p) - 2 R_{i\bar{j}k\bar{\ell}}^M(p) + 2\sqrt{-1} E_i h_{jk\ell}(p)\right)x_ix_jx_kx_{\ell} \\
& \quad + \frac{\sqrt{-1}}{6}\sum_{i, j, k,\ell=1}^n h_{ijk}(p)E_{\ell} f(p) x_ix_jx_kx_{\ell} -\frac{f(p)}{72} \sum_{i, j, k, \ell, m, t=1}^n h_{ijk}(p)h_{\ell m t}(p) x_i x_j x_k x_{\ell}x_m x_t.    
\end{align*}
Since 
$
\frac{1}{(2\pi)^{n/2}}\int_{\mathbb{R}^n} e^{-\frac{|x|^2}{2}} x_ix_j dx_1 \cdots dx_n = \delta_{ij}, 
$
we have 
\begin{align*}
&\sum_{i, j=1}^n  \frac{1}{(2\pi)^{n/2}} \int_{\mathbb{R}^n} e^{-\frac{|x|^2}{2}} 
\left(\frac{1}{2}E_i E_j f(p) - \frac{1}{6}f(p) \sum_{m = 1}^n R_{m ij m}^{\Lambda}(p) \right)x_i x_j dx_1 \cdots dx_n \\
= & -\frac{1}{2}\Delta_{\Lambda} f(p) -  \frac{f(p)}{6} \mathrm{Scal}_{\Lambda} (p). 
\end{align*}
Since 
$
\frac{1}{(2\pi)^{n/2}} \int_{\mathbb{R}^n} e^{-\frac{|x|^2}{2}} x_ix_jx_kx_{\ell} dx_1 \cdots dx_n = \delta_{ij} \delta_{k \ell} + \delta_{ik}\delta_{j \ell} + \delta_{i \ell} \delta_{jk}, 
$
we have 
\begin{align*}
& \frac{1}{(2\pi)^{n/2}} \sum_{i, j, k, \ell=1}^n \int_{\mathbb{R}^n} e^{-\frac{|x|^2}{2}} \sum_{m=1}^n h_{ikm} h_{j\ell m} x_ix_jx_kx_{\ell} dx_1 \cdots dx_n \\
= & \sum_{m=1}^n \sum_{i, j=1}^n (2 h_{ijm}(p)^2 + h_{iim}(p)h_{jjm}(p)) 
= 2 |B(p)|^2 + |H(p)|^2,   
\end{align*}
\begin{align*}
& \frac{1}{(2\pi)^{n/2}} \sum_{i, j, k, \ell=1}^n \int_{\mathbb{R}^n} e^{-\frac{|x|^2}{2}} R_{i\bar{j}k\bar{\ell}}^M x_ix_jx_kx_{\ell} dx_1 \cdots dx_n \\
= & \sum_{i, j=1}^n (R_{i\bar{j}j\bar{i}}^M(p) + R_{i \bar{j} i \bar{j}}^M(p) + R_{i \bar{i} j \bar{j}}^M(p))
= \frac{3}{8} \mathrm{Scal}_M(p) -\frac{1}{2} \mathrm{Scal}_{\Lambda}(p) + \frac{1}{2} |H(p)|^2 -\frac{1}{2}|B(p)|^2 
\end{align*}
by Lemma~\ref{lemma:3}, and 
\begin{align*}
& \frac{1}{(2\pi)^{n/2}} \sum_{i, j, k, \ell=1}^n \int_{\mathbb{R}^n} e^{-\frac{|x|^2}{2}}  E_i h_{jk \ell}(p)  x_ix_jx_kx_{\ell} dx_1 \cdots dx_n 
=  3 \sum_{i, j =1}^n E_i h_{ijj}(p).  
\end{align*}
Now, $JH = - \sum_{i, j=1}^n h_{ijj}E_{i} + O(|x|^2)$.  
Hence, we have $\sum_{i, j=1}^n E_{i} h_{ijj}(p) = -\mathrm{div}_{\Lambda} (JH)(p)$.  

We also have 
\begin{align*}
\frac{1}{(2\pi)^{n/2}}\sum_{i, j, k, \ell = 1}^n \int_{\mathbb{R}^n} e^{-\frac{|x|^2}{2}} h_{ijk}(p) E_{\ell}f(p) x_ix_jx_kx_{\ell} dx_1 \cdots dx_n 
= 3\sum_{i, j=1}^n h_{iij} E_{j} f(p) = - 3JH f (p).  
\end{align*}

It remains only to evaluate the sixth--order Gaussian term. 
The Gaussian 
\[
\frac{1}{(2\pi)^{n/2}}\int_{\mathbb{R}^n} e^{-\frac{|x|^2}{2}}x_ix_jx_kx_{\ell}x_m x_t dx_1 \cdots dx_n
\]
is the sum of the products of Kronecker deltas corresponding to all pairings of the six indices.
There are $15$ such pairings.
For the contraction with $h_{ijk}h_{\ell mt}$, there are $6$ pairings in which all three pairs connect the two triples $(i,j,k)$ and $(\ell,m,t)$, and $9$ pairings in which only one pair connects the two triples.
Hence, we have 
\begin{align*}
& \frac{1}{(2\pi)^{n/2}}\sum_{i, j, k, \ell, m, t=1}^n \int_{\mathbb{R}^n} e^{-\frac{|x|^2}{2}} h_{ijk}(p) h_{\ell m t}(p) x_ix_jx_kx_{\ell}x_mx_t dx_1 \cdots dx_n \\
= & \sum_{i, j, k=1}^n( 6h_{ijk}^2(p) + 9 h_{iik}(p) h_{jjk}(p))
= 6|B(p)|^2 + 9|H(p)|^2.  
\end{align*}
By combining these equations, we obtain 
\begin{align*}
& \frac{1}{(2\pi)^{n/2}} \int_{\mathbb{R}^n} e^{-\frac{|x|^2}{2}}Q_{p, 2}(0, x) dx_1 \cdots dx_n \\
= & \left(\frac{3}{32} \mathrm{Scal}_M(p) -\frac{1}{8}\mathrm{Scal}_{\Lambda}(p)+ \frac{1}{24} |B(p)|^2 - \frac{1}{8}|H(p)|^2 \right) f(p) - \frac{1}{2} \Delta_{\Lambda} f(p) \\
& \quad- \sqrt{-1} \left(\frac{1}{4} \mathrm{div}_{\Lambda} JH(p) f(p) + \frac{1}{2} JH f (p) \right).   
\end{align*}
By the definition of $c_{p,2}(0)$, the left-hand side is $c_{p,2}(0)$.
This proves Theorem~\ref{theorem:2}.
\begin{proof}[Proof of Corollary~\ref{corollary:1}]
Since $s_{f, k} = \left(\frac{\pi}{2k}\right)^{n/2} R_k^* f$, 
we have 
\begin{align*}
& \left(\frac{2k}{\pi} \right)^{n/2} \|s_{f, k}\|^2_{h^k} 
= \int_{\Lambda} \frac{s_{f, k}}{\zeta^k} f dv_{\Lambda} \\
= &   
\int_\Lambda f^2\,dv_\Lambda 
+
\frac{1}{k}
\int_\Lambda
\left[
\left(
\frac{3}{32}\mathrm{Scal}_M
-\frac{1}{8}\mathrm{Scal}_\Lambda
+\frac{1}{24}|B|^2
-\frac{1}{8}|H|^2
\right)f^2
- \frac{1}{2} f\Delta_{\Lambda}f  
\right]dv_\Lambda \\
&\quad - \frac{\sqrt{-1}}{4k} \int_{\Lambda}\left(\mathrm{div}_{\Lambda}(JH) f^2 + 2fJH f \right) dv_{\Lambda}  
+
O\left(\frac{1}{k^2}\right).
\end{align*}
Since $\left(\frac{2k}{\pi} \right)^{n/2} \|s_{f, k}\|^2_{h^k} $ is real-valued, the imaginary part 
must vanish. 
However, this can also be verified directly from 
\[
\int_{\Lambda}\left(\mathrm{div}_{\Lambda}(JH) f^2 + 2f JH f \right) dv_{\Lambda} = 
\int_{\Lambda} \mathrm{div}_{\Lambda} (f^2 JH) dv_{\Lambda} = 0.  
\]
\end{proof}

\section{Examples}\label{section:6}
In this section, we give some examples in $M=\mathbb{CP}^n$.
Let $[X_0:X_1:\cdots:X_n]$ be the homogeneous coordinates on $\mathbb{CP}^n$.
Let
$
U=\{X_0\neq 0\}\simeq\mathbb{C}^n
$
be the standard affine chart, and let $z=(z_1,\ldots,z_n)$ be the holomorphic coordinates on $U$ defined by
$
z_j=\frac{X_j}{X_0}.
$
The standard Fubini--Study K\"ahler form on $\mathbb{CP}^n$ is given by
\[
\omega_{FS}
=
\frac{\sqrt{-1}}{2}\partial\bar\partial\log(1+|z|^2)
=
\frac{\sqrt{-1}}{2}
\sum_{i,j=1}^n
\left(
\frac{\delta_{ij}}{1+|z|^2}
-
\frac{\bar z_i z_j}{(1+|z|^2)^2}
\right)
dz_i\wedge d\bar z_j.
\]
We use the superscript or subscript $FS$ to indicate geometric quantities computed with respect to the Fubini--Study metric; for example, we write $\mathrm{Scal}_M^{FS}$ and $\langle\cdot,\cdot\rangle_{FS}$.
We also note that the Riemannian curvature operator is unchanged under constant rescaling of the metric.

At $z=0$, we have
\begin{align*}
\left\langle
R^{M}(\partial_i,\bar\partial_j)\partial_k,
\bar\partial_\ell
\right\rangle_{FS}
=
\left. -\frac12
\partial_i\bar\partial_j
\left(
\frac{\delta_{k\ell}}{1+|z|^2}
-
\frac{\bar z_k z_\ell}{(1+|z|^2)^2}
\right)\right|_{z=0}
=
\frac12
\left(
\delta_{ij}\delta_{k\ell}
+
\delta_{jk}\delta_{i\ell}
\right).
\end{align*}
Since the Fubini--Study metric is invariant under the transitive action of $U(n+1)$ on $\mathbb{CP}^n$, 
the scalar curvature $\mathrm{Scal}_M^{FS}$ is constant on $\mathbb{CP}^n$ and 
\[
\mathrm{Scal}_{M}^{FS} = 8\sum_{i, j=1}^n \langle R^{M}(\partial_i,\bar\partial_j)\partial_j,
\bar\partial_i\rangle_{FS}
= 4n(n+1).  
\]
Next, we recall the Bergman kernel of $(\mathcal{O}(k), h_{FS}^k)$ with respect to the volume form $\frac{\omega_{FS}^n}{n!}$, $k \in \mathbb{N}$, on $\mathbb{CP}^n$.  
We identify $H^0(\mathbb{CP}^n,\mathcal{O}(k))$ with the space of homogeneous
polynomials of degree $k$ on $\mathbb C^{n+1}$. 
Let
$X= (X_0, \ldots, X_n), Y= (Y_0, \ldots, Y_n)$ be the coordinates
of the first and second variables, respectively.  
The Bergman kernel can be
represented by a function which is homogeneous of degree $k$ in $X$ and
anti-homogeneous of degree $k$ in $Y$.
Then, by the $U(n+1)$-invariance, the Bergman kernel of $\mathcal O(k)$ is of the form
\[
K_k^{FS}(X,Y)
=
C_{k,n}(X\cdot \overline{Y})^k,
\qquad
X\cdot \overline{Y}
=
\sum_{j=0}^n X_j\overline{Y}_j.
\]
We have 
\[
\int_{\mathbb{CP}^n} |K_{k}^{FS}(X, X)|_{h^{k}_{FS}} \frac{\omega_{FS}^n}{n!} 
=
\dim H^0(\mathbb{CP}^n,\mathcal O(k)).
\]
Since 
\[
\frac{\omega_{FS}^n}{n!} = \frac{
\frac{\sqrt{-1}^n}{2^n} dz_1\wedge d\bar{z}_{1} \cdots dz_n \wedge d\bar{z}_n}{(1 + |z|^2)^{n+1}}, 
\]
on $U$, 
it follows that 
$\int_{\mathbb{CP}^n} \frac{\omega^n_{FS}}{n!} = \frac{\pi^n}{n!}$.  
Using 
$
\dim H^0(\mathbb{CP}^n,\mathcal O(k))
= \binom{n+k}{n}
$
we obtain
\[
K_k^{FS}(X,Y) 
=
\frac{(n+k)!}{\pi^n k!}
(X\cdot\overline{Y})^k.
\]
\vspace{2mm}\\
(1) 
We consider the real projective space
$
\Lambda=\mathbb{RP}^n\subset\mathbb{CP}^n.
$
We take
$(L,h)=(\mathcal O(2),h_{FS}^2)$,
$\omega=2\omega_{FS}$.

Since the corresponding Riemannian metric is twice the Fubini--Study metric,
we have 
$
\mathrm{Scal}_M=\frac12\mathrm{Scal}_M^{FS} = 2n(n+1).  
$
Moreover,
the Bergman kernel of $(L^k,h^k)$ with respect to
$\frac{\omega^n}{n!}$ is 
\[
K_k(X,Y)
=
\frac{(n+2k)!}{(2\pi)^n(2k)!}
(X\cdot\overline{Y})^{2k}.
\]
We compute $\mathrm{Scal}^{FS}_\Lambda$.
Let
$
\pi:S^n\longrightarrow \mathbb{RP}^n
$
be the standard double covering. The pullback by $\pi$ of the
Fubini--Study metric restricted to $\mathbb{RP}^n$ is the standard round
metric on $S^n$ of sectional curvature $1$. 
Hence
$
\mathrm{Scal}^{FS}_\Lambda=n(n-1)
$
and 
$\mathrm{Scal}_\Lambda = \frac{1}{2}\mathrm{Scal}_{\Lambda}^{FS} = \frac{n(n-1)}{2}$.
Since $\Lambda$ is the fixed point set of the anti-holomorphic isometry
\[
\tau:\mathbb{CP}^n\longrightarrow\mathbb{CP}^n,
\qquad
[X_0:\cdots:X_n]\longmapsto
[\overline{X}_0:\cdots:\overline{X}_n], 
\]
$\Lambda$ is a totally geodesic Lagrangian submanifold.
Hence
$B=0$, $H=0$.
The restriction of the holomorphic section
$
\zeta = X_0^2+\cdots+X_n^2
$
of $L$ to $\Lambda$ is a unit parallel section.
Hence, $\Lambda$ is the Bohr--Sommerfeld Lagrangian submanifold.  

Now we take $f = 1 \in C^{\infty}(\Lambda)$.  
At $p = [1:0:\ldots:0] \in \Lambda$, we have 
\[
\frac{s_{1, k}}{\zeta^k}(p) = \left(\frac{\pi}{2k}\right)^{n/2} \frac{(n+2k)!}{(2\pi)^n (2k)!}\frac{1}{2}\int_{S^{n}} Y_{0}^{2k}\, 2^{\frac{n}{2}}dg_{S^{n}}.  
\]
Here $dg_{S^n}$ is the density on $S^n$ induced by the standard round metric.  
To compute the integral, we use the polar coordinates
\[
(0,\pi)\times S^{n-1}\longrightarrow S^n,
\qquad
(\theta,q)\longmapsto(\cos\theta,\sin\theta\,q).
\]
In these coordinates,
$
g_{S^n}
=
d\theta^2+\sin^2\theta\,g_{S^{n-1}},
$
and hence
$
dg_{S^n}
=
(\sin\theta)^{n-1}\,d\theta\,dg_{S^{n-1}}.
$
Here, $g_{S^{n-1}}$ is the standard round metric on $S^{n-1}$.  
We obtain
\begin{align*}
\int_{S^n}Y_0^{2k}\,dg_{S^n}
=
\operatorname{Vol}(S^{n-1})
\int_0^\pi
\cos^{2k}\theta\,\sin^{n-1}\theta\,d\theta 
= \frac{2\pi^{\frac{n}{2}}}{\Gamma\left(\frac{n}{2}\right)} B\left(k+\frac{1}{2}, \frac{n}{2}\right) 
= \frac{2\pi^{\frac{n}{2}}\Gamma\left(k+\frac{1}{2}\right)}{\Gamma\left(k + \frac{n+1}{2}\right)}.
\end{align*}
For $a, b > 0$ we have the asymptotic formula 
\[
\frac{\Gamma(x+a)}{\Gamma(x+b)} = x^{a-b} \left(1 + \frac{(a-b)(a+b-1)}{2x} + O\left(\frac{1}{x^2}\right)\right) \qquad  (x \longrightarrow +\infty)
\]
(cf.\, Chapter~3 of \cite{Tem}).  
By using this formula, we obtain 
\begin{align*}
\frac{s_{1, k}}{\zeta^k}(p) & 
= \left(1 + \frac{n(n+1)}{4k} + O(k^{-2}) \right) \left(1-\frac{n^2}{8k} + O(k^{-2})\right)
= 1 + \frac{n(n+2)}{8k} + O(k^{-2}).  
\end{align*}
On the other hand, 
by substituting $\mathrm{Scal}_M = 2n(n+1)$, $\mathrm{Scal}_{\Lambda} = \frac{n(n-1)}{2}$, $B=H=0$ and $f = 1$ into Theorem~\ref{theorem:2}, we obtain the same coefficient.  
\vspace{3mm}\\
(2)
We consider the Clifford Torus $\Lambda = \{[X_0: \ldots :X_n] \in \mathbb{CP}^n \mid |X_0| = \cdots = |X_n|\}$.  
We take $(L, h) = (\mathcal{O}(n+1), h_{FS}^{n+1})$, $\omega = (n+1)\omega_{FS}$.  
Then, $\Lambda$ is a Lagrangian submanifold of $(M, \omega)$.  
The restriction of the holomorphic section $\zeta = (n+1)^{(n+1)/2} X_0 \cdots X_n$ of $L$ to $\Lambda$ is a unit parallel section.  
Hence, $\Lambda$ is the Bohr--Sommerfeld submanifold.  

In this case, $\mathrm{Scal}_M = \frac{1}{n+1} \mathrm{Scal}_{M}^{FS} = 4n$, and 
the Bergman kernel of $(L^k, h^k)$ with respect to $\frac{\omega^n}{n!}$ is 
\[
K_{k}(X, Y) = \frac{(kn+k+n)!}{(n+1)^n\pi^n (kn+k)!} (X\cdot \overline{Y})^{(n+1)k}.   
\]

The induced metric on $\Lambda$ is invariant under the translation action
of the torus. 
Hence, this metric is flat, and $\mathrm{Scal}_{\Lambda} = 0$.  
Let $\mathcal S\simeq\mathfrak S_{n+1}$ be the subgroup of $U(n+1)$
consisting of permutation matrices. It acts on $\mathbb{CP}^n$ by
\[
\sigma\cdot[X_0:\cdots:X_n]
=
[X_{\sigma(0)}:\cdots:X_{\sigma(n)}],
\qquad
\sigma\in\mathcal S.
\]
This action preserves both the Fubini--Study metric and $\Lambda$. 
Moreover, the point
$
p=[1:\cdots:1]\in\Lambda
$
is fixed by $\mathcal S$.
Hence $JH(p)$ is invariant under the action of $\mathcal S$, and we obtain $JH(p) = 0$.  
Therefore, we have $H = 0$ on $\Lambda$.  
Now we have $R_{i\bar{j}i\bar{j}}^M = \frac{\delta_{ij}}{n+1}$.  
Then, by Lemma~\ref{lemma:3}, it follows that 
$|B|^2 = \frac{n(n-1)}{n+1}$.  
Using the holomorphic coordinates $z=(z_1,\ldots,z_n)$ on $U$,
we define coordinates $\theta=(\theta_1,\ldots,\theta_n)$ on $\Lambda$ by
$
z_j=e^{\sqrt{-1}\theta_j}, j=1,\ldots,n.
$
Then 
\[
g_{\Lambda} = \sum_{i, j=1}^n\left(\delta_{ij}-\frac{1}{n+1}\right)d\theta_i d\theta_j
\]
and $dg_{\Lambda} = \frac{1}{\sqrt{n+1}}d\theta_1 \cdots d\theta_n$.  

We take $f = 1 \in C^{\infty}(\Lambda)$.  
At $p = [1:\ldots:1] \in \Lambda$, we have 
\begin{align*}
\frac{s_{1, k}}{\zeta^k}(p) & = \left(\frac{\pi}{2k}\right)^{n/2} \frac{(kn+k+n)!}{(n+1)^{n+k(n+1)+1/2}\pi^n (kn+k)!}\\
& \quad \times \int_{[0, 2\pi]^n} \frac{(1+e^{-\sqrt{-1} \theta_1} + \cdots + e^{-\sqrt{-1}\theta_n})^{k(n+1)}}{e^{-k(\sqrt{-1} \theta_1 + \cdots + \sqrt{-1} \theta_n)}} d\theta_1 \cdots d\theta_n \\
& = 
\left(\frac{2\pi}{k}\right)^{n/2}\frac{(kn+k+n)!}{(n+1)^{n+k(n+1)+1/2} (k!)^{n+1}}.  
\end{align*}
By Stirling's formula (see, e.g., \cite{Tem}), we have 
\[
\Gamma(x) = \sqrt{2\pi}x^{x-1/2} e^{-x} \left(1 + \frac{1}{12x} + O\left(\frac{1}{x^2}\right)\right)   \qquad (x \longrightarrow +\infty).  
\]
By putting $a = n+1$, $x = k+1$, we obtain 
\begin{align*}
\frac{s_{1, k}}{\zeta^k}(p) & 
= \left(\frac{2\pi}{k}\right)^{n/2} \frac{\Gamma(ax)}{a^{ax-1/2} \Gamma(x)^{a}} 
= \frac{1}{k^{n/2}} x^{(a-1)/2} \left(1+\frac{1}{12ax}+O\left(\frac{1}{x^2}\right)\right) \left(1-\frac{a}{12x}+O\left(\frac{1}{x^2}\right)\right) \\
& = \left(1+\frac{1}{k}\right)^{n/2} \left(1-\frac{n^2+2n}{12(n+1)(k+1)} + O\left(\frac{1}{k^2}\right)\right)
=1+ \frac{n(5n+4)}{12(n+1)} \frac{1}{k} + O\left(\frac{1}{k^2}\right).  
\end{align*}
On the other hand, by substituting $\mathrm{Scal}_M = 4n$, $\mathrm{Scal}_{\Lambda} = 0$, $H = 0$, $|B|^2 = \frac{n(n-1)}{n+1}$ and $f=1$ into Theorem~\ref{theorem:2}, we obtain the same coefficient.

\vspace{5mm}

\par\noindent{\scshape \small
Department of Mathematics, \\
Ochanomizu University,  \\
2-1-1 Otsuka, Bunkyo-ku, Tokyo (Japan) }
\par\noindent{\ttfamily chiba.yusaku@ocha.ac.jp}
\end{document}